\documentclass[12pt]{article}
\usepackage{amsthm}

\usepackage{times}
\usepackage{newtxmath}
\usepackage[margin=1.0in]{geometry}

\newtheorem{conjecture}{Conjecture}

\newtheorem{theorem}{Theorem}

\theoremstyle{definition}
\newtheorem{reference}{Reference}
\newtheorem{definition}{Definition}
\newtheorem{remark}{Remark}
\newtheorem{assumption}{Assumption}

\begin{document}

\title{$\infty$-Categorical Generalized Langlands Correspondence III: $\infty$-Stackification}
\author{Xin Tong}
\date{}

\maketitle

\begin{abstract}
\noindent We discuss further around the generalized Langlands Program, by using $\infty$-categoricalization and $\infty$-analytic stackification.
\end{abstract}

\newpage

\tableofcontents

\newpage
\begin{reference}
A: \cite{La}, \cite{DHKM}, \cite{DA}, \cite{DB}, \cite{VL}, \cite{DK}, \cite{A}, \cite{XZ}, \cite{GL}, \cite{K1}, \cite{KX}, \cite{BS}, \cite{FS}, \cite{Z}, \cite{EGH}, \cite{CA}; B: \cite{KLA}, \cite{KLB}, \cite{SchA}, \cite{SchB}, \cite{SchC}, \cite{FA}, \cite{Ta}, \cite{BS}, \cite{KPX}; C: \cite{CSA}, \cite{CSB}, \cite{CSC}, \cite{BBBK}; D: \cite{BS2}, \cite{BLA}, \cite{DC}, \cite{SchD}, \cite{SALBRC}, \cite{TA}, \cite{TB}, \cite{TC}. 
\end{reference}

\begin{assumption}
Let $n\geq 2$. Here $n$ will be a key parameter for our consideration. We will have $\mathrm{Gal}_{X,n}$, $\mathrm{Weil}_{X,n}$ and all the related objects parametrized by $n$ related to the mixed-parity generalization. We will also fix the prime number $p$ and $\ell$, with the local nonarchimedean field $X$, where $\ell$ is the coefficient characteristic, we make the requirement on $\overline{\mathbb{Q}}_\ell$ as in \cite{FS} with respect to the corresponding $|\mathcal{O}_X/\mathfrak{m}_{\mathcal{O}_X}|$.
\end{assumption}

\newpage

\section{Condensation Conjectures}

\subsection{Condensation Conjectures}

\noindent Although Langlands conjecture is slightly relatively more  representation theoretic in \cite{La}, many of the current existing Langlands parametrization happen over certain moduli stacks, such as the group double quotients for the number fields, the group double quotients for the local fields, where the corresponding $\ell$-adic perverse sheaves will capture the action from the motivic Galois groups and the Langlands dual groups, i.e. a map to the Bernstein centers. Condensation is needed in the local setting, and especially when we are using certain analytic methods, such as Scholze's diamonds and even the corresponding analytic cohomologies. The current paper is along certain direction generalizing significantly the paper \cite{BS} in the following sense. First we consider finite extension of a local nonarchimedean field $X$, which is assumed to be a finite extension of the $p$-adic number field or a local function field such as $\mathbb{F}_p((z))$. Over $X$ there are many considerations can be made beyond the consideration in \cite{BS}. \cite{BS} actually considered certain condensed morphism in the $p$-adic consideration:
\begin{align}
\mathrm{Gal}_{X,2} \rightarrow G^\mathrm{Lan}(\overline{T})
\end{align}
where $T/\overline{\mathbb{Q}}_p$ is a corresponding $p$-adic algebraic closure, and $X$ is assumed to be $p$-adic. The first thing we want to consider is the following $z$-adic generalization:

\begin{conjecture}
For the general $X$ as in the above in the $z$-adic situation, we conjecture there exists certain \textit{Breuil-Schneider Conjecture}, namely we conjecture there exists certain condensed parametrization 
\begin{align}
\mathrm{Gal}_{X,2} \rightarrow G^\mathrm{Lan}(\overline{T})
\end{align}
for generalized $z$-adic Banach representations for $z$-adic reductive group $G(X)$.
\end{conjecture}

\begin{remark}
Indeed the author does \textit{not} know the definition of generalized $z$-adic Banach representations for $z$-adic reductive group $G(X)$, but we conjecture there is such a way to define this for instance after \cite{EGH}. Here we assume $T$ is finite extension of $\mathbb{Q}_p$ or $\mathbb{F}((z))$, which is sufficiently large by assumption to be able to be used to defined the coverings of Galois groups in this paper and all the corresponding roots of cyclotomic charaters, and the action of the cyclotomic characters and the Frobenius operators.
\end{remark}

\begin{conjecture}
For $n$-fold covering of Galois group $\mathrm{Gal}_{X,n}$ by taking the roots of cyclotomic charaters up to order $n$. For the general $X$ as in the above in the $z$-adic situation, we conjecture there exists certain \textit{Breuil-Schneider Conjecture}, namely we conjecture there exists certain condensed parametrization 
\begin{align}
\mathrm{Gal}_{X,n} \rightarrow G^\mathrm{Lan}(\overline{T})
\end{align}
for generalized $z$-adic Banach representations for $z$-adic reductive group $G(X)$.
\end{conjecture}

\indent When we have $\ell$-adic coefficient actually we also have certain conjectures in some obvious sense:

\begin{conjecture}
There is a way to generalize \cite{FS} to our setting, namely for any $n$-fold covering of Galois group $\mathrm{Gal}_{X,n}$ by taking the roots of cyclotomic charaters up to order $n$, we have certain condensed parametrization:
\begin{align}
\mathrm{Gal}_{X,n} \rightarrow G^\mathrm{dual}(\overline{\mathbb{Q}}_\ell)\times \mathrm{Weil}_X
\end{align}
for generalized $\ell$-adic perverse sheaves attached to $G(X)$. One can also replace the full Galois group with the corresponding Weil group version. For any $n$-fold covering of Galois group $\mathrm{Gal}_{X,n}$ by taking the roots of cyclotomic charaters up to order $n$, we have certain condensed parametrization:
\begin{align}
\mathrm{Weil}_{X,n} \rightarrow G^\mathrm{dual}(\overline{\mathbb{Q}}_\ell)\times \mathrm{Weil}_X
\end{align}
for generalized $\ell$-adic perverse sheaves attached to $G(X)$. 
\end{conjecture}

\noindent After \cite{EGH} we conjecture that there is also certain way to generalize the rigid analytic stack from \cite{EGH} to our setting. First we replace Breuil-Schneider consideration with the Robba ring version of the consideration. 

\begin{conjecture}
For the general $X$ as in the above in the $z$-adic situation, we conjecture there exists certain \textit{Breuil-Schneider Conjecture}, namely we conjecture there exists certain condensed parametrization of $(\varphi,\Gamma_n)$-$G^\mathrm{Lan}(\overline{T})$-bundle over the generalized Robba ring $\Pi_X[\log(1+T)^{1/n}]$ in the imperfect setting, for generalized $p$-adic Banach representations for $z$-adic reductive group $G(X)$ and more. For the general $X$ as in the above in the $p$-adic situation, we conjecture there exists certain \textit{Breuil-Schneider Conjecture}, namely we conjecture there exists certain condensed parametrization of $(\varphi,\Gamma_n)$-$G^\mathrm{Lan}(\overline{T})$-bundle over the generalized Robba ring $\Pi_X[\log(1+T)^{1/n}]$ in the imperfect setting, for generalized $p$-adic Banach representations for $p$-adic reductive group $G(X)$ and more.
\end{conjecture}

\begin{remark}
Here we assume $T$ is finite extension of $\mathbb{Q}_p$ or $\mathbb{F}((z))$, which is sufficiently large by assumption to be able to be used to defined the coverings of Galois groups in this paper and all the corresponding roots of cyclotomic charaters, and the action of the cyclotomic characters and the Frobenius operators.
\end{remark}

\indent In this paper we also consider the corresponding geometrization of Breuil-Schneider's conjecture in \cite{BS}, by using the prismatization, after \cite{CSA}, \cite{CSB}, \cite{CSC}, \cite{SchD}, \cite{SALBRC}, \cite{TA}, \cite{TB}, \cite{TC}, \cite{BBBK}, \cite{BS2}, \cite{BLA}, \cite{DC}. Essentially we needs the \textit{condensation}, which might provide directly the Banach norms in \cite{BS} automatically.

\begin{conjecture}
The Breuil-Schneider style functor 
\begin{align}
&\otimes^\blacksquare: \mathrm{QuasiCohSheaves}^\blacksquare_{\mathrm{prismatic},\mathrm{Stack}_{\mathrm{prismatic}}}(\mathrm{Stack}_{\mathrm{Bun},G,X,n})_{\mathrm{End}}\\
&\times \mathrm{QuasiCohSheaves}^\blacksquare_{\mathrm{prismatic},\mathrm{Stack}_{\mathrm{prismatic}}}(\mathrm{Stack}_{\mathrm{Bun},G,X,n})_{\mathrm{Loc}}\\
& \rightarrow \mathrm{QuasiCohSheaves}^\blacksquare_{\mathrm{prismatic},\mathrm{Stack}_{\mathrm{prismatic}}}(\mathrm{Stack}_{\mathrm{Bun},G,X,n})
\end{align}
can be used to construct geometrization of the Breuil-Schneider paramatrization in \cite{BS}.
\end{conjecture}

\subsection{Generalized Langlands Parameter Stackification}

\indent After \cite{EGH} we immediately have the following stacks:

\begin{definition}
$\mathrm{Stack}_{(\varphi,\Gamma_n), G^\mathrm{Lan}(\overline{T}),\Pi_X[\log(1+T)^{1/n}]}$ are stacks over rigid analytic spaces $\mathrm{RAS}_X$, parametrizing arithmetically all the $(\varphi,\Gamma_n)$-$G^\mathrm{Lan}(\overline{T})$-bundles over the generalized Robba ring $\Pi_X[\log(1+T)^{1/n}]$ in the imperfect setting.
\end{definition}

\indent Then one can follow \cite{FS}, \cite{DHKM}, \cite{EGH}, \cite{Z} to consider the corresponding coherent sheaves over these stacks by using the points coming from rigid analytic spaces, even in the derived sense by animating the corresponding rigid analytic affinoids from \cite{CSA}, \cite{CSB}, \cite{CSC}, where we use the same notation to denote the corresponding $\infty$-stacks.

\begin{definition}
$\mathrm{Stack}^\blacksquare_{(\varphi,\Gamma_n), G^\mathrm{Lan}(\overline{T}),\Pi_X[\log(1+T)^{1/n}]}$ are $\infty$-stacks over derived rigid analytic spaces $\mathrm{RAS}^\blacksquare_X$, parametrizing arithmetically all the $(\varphi,\Gamma_n)$-$G^\mathrm{Lan}(\overline{T})$-bundles over the generalized Robba ring 
\begin{align}
\Pi_{X,\blacksquare}[\log(1+T)^{1/n}]
\end{align}
in the imperfect setting. We then have the corresponding coherent sheaves over these $\infty$-stacks, where we use the notation $\mathrm{Coh}_\blacksquare(\mathrm{Stack}^\blacksquare_{(\varphi,\Gamma_n), G^\mathrm{Lan}(\overline{T}),\Pi_X[\log(1+T)^{1/n}]})$ to denote the corresponding condesed coherent sheaves.
\end{definition}

\subsection{Results}

\indent We introduce now the results of this paper. We generalize the corresponding context of \cite{FS} to certain generalized context where extension of Galois actions are introduced. The corresponding functorial deformation of Robba sheaves give rise to certain analytic stackification of the corresponding filtred Hodge structures in the corresponding prismatic/crystalline situations. We then have the chance to construct the corresponding parametrization through stackification of the generalized $\ell$-adic perverse sheaves generalizing the smooth representations of reductive nonarchimedean groups. The coefficient can be $p$-adic as well, where one can use the corresponding prismatization approach following \cite{CSA}, \cite{CSB}, \cite{CSC}, \cite{SchD}, \cite{SALBRC}, \cite{TA}, \cite{TB}, \cite{TC}, \cite{BBBK}, \cite{BS2}, \cite{BLA}, \cite{DC}. This can be use to geometrize Breuil-Schneider's original conjecture in \cite{BS}.

\begin{theorem}
$\mathrm{Stack}_{\mathrm{Bun},G,X,n}(\square)$ for $n\geq 2$ are actually equal to there stackification over the $v$-sites. And they satisfy the condition of being of smallness. Finally we can prove that they are satisfying the condition of being Artin as in \cite{FS}.
\end{theorem}

\begin{theorem}
\begin{align}
&\mathrm{FilFibCat}_{\mathrm{crys},n}(\square),\\
&\mathrm{FilFibCat}_{\mathrm{crys},n,\mathrm{Frob}}(\square),
\end{align} 
are stacks, with certain morphisms into:
\begin{align}
\mathrm{Stack}_{\mathrm{Bun},G,X,n}.
\end{align}
\end{theorem}

\begin{theorem}
Over the stack $\mathrm{Stack}_{\mathrm{Bun},G,X,n}$ we have the corresponding $v$-stacks of categories of prismatic quasi-coherent sheaves:
\begin{align}
\mathrm{QuasiCohSheaves}^\blacksquare_{\mathrm{prismatic},\mathrm{Stack}_{\mathrm{prismatic}}}(\mathrm{Stack}_{\mathrm{Bun},G,X,n}).
\end{align}
These quasi-coherent sheaves can be used to define the $p$-adic motives over category $\mathrm{Stack}_{\mathrm{Bun},G,X,n}$. Moreover we have the well-defined category of all the objects have the requirement $\mathrm{End}$ is a line:
\begin{align}
\mathrm{QuasiCohSheaves}^\blacksquare_{\mathrm{prismatic},\mathrm{Stack}_{\mathrm{prismatic}}}(\mathrm{Stack}_{\mathrm{Bun},G,X,n}).
\end{align}
They cover the objects related to lisse $p$-adic representations of the reductive groups.
\end{theorem}

\begin{theorem}
Over the stack $\mathrm{Stack}_{\mathrm{Bun},G,X,n}$ we have the corresponding $v$-stacks of categories of prismatic quasi-coherent sheaves:
\begin{align}
\mathrm{QuasiCohSheaves}^\blacksquare_{\mathrm{prismatic},\mathrm{Stack}_{\mathrm{prismatic},\mathrm{deRham}}}(\mathrm{Stack}_{\mathrm{Bun},G,X,n}).
\end{align}
These quasi-coherent sheaves can be used to define the $p$-adic motives over category $\mathrm{Stack}_{\mathrm{Bun},G,X,n}$. Moreover we have the well-defined category of all the objects have the requirement $\mathrm{End}$ is a line:
\begin{align}
\mathrm{QuasiCohSheaves}^\blacksquare_{\mathrm{prismatic},\mathrm{Stack}_{\mathrm{prismatic}}}(\mathrm{Stack}_{\mathrm{Bun},G,X,n}).
\end{align}
They cover the objects related to lisse $p$-adic representations of the reductive groups.
\end{theorem}

\begin{theorem}
There is a way to generalize \cite{FS} to our setting, namely for any $n$-fold covering of Galois group $\mathrm{Gal}_{X,n}$ by taking the roots of cyclotomic charaters up to order $n$, we have certain condensed parametrization:
\begin{align}
\mathrm{Gal}_{X,n} \rightarrow G^\mathrm{dual}(\overline{\mathbb{Q}}_\ell)\times \mathrm{Weil}_X
\end{align}
for generalized $\ell$-adic perverse sheaves attached to $G(X)$, which make $\mathrm{End}$ a line. One can also replace the full Galois group with the corresponding Weil group version. For any $n$-fold covering of Galois group $\mathrm{Gal}_{X,n}$ by taking the roots of cyclotomic charaters up to order $n$, we have certain condensed parametrization:
\begin{align}
\mathrm{Weil}_{X,n} \rightarrow G^\mathrm{dual}(\overline{\mathbb{Q}}_\ell)\times \mathrm{Weil}_X
\end{align}
for generalized $\ell$-adic perverse sheaves attached to $G(X)$, which make $\mathrm{End}$ a line. 
\end{theorem}

\begin{remark}
Here the corresponding de Rham stack over any perfectoid $\square$ is defined to be the formal spectrum of the corresponding de Rham period sheaf attached to an untilt $\sharp\square$ of $\square$, since we are working in an absolute situation where we do not fix some prism as the base ring. To be more precise we use the corresponding stack 
\begin{align}
\mathrm{Specformal}\Pi_\mathrm{deRham}(\sharp\square^\flat)=\mathrm{Specformal}\mathrm{WittVector}_\mathcal{O}(\sharp\square^\flat)[1/p]_{I_{\sharp\square}}. 
\end{align}
Then we take the corresponding \textit{condensed analytification}. If we have a fixed prism as the base ring then we can just use the Cartier divisor for this ring to take the completion with respect to this divisor after inverting $p$. In this manner is $X$ is $z$-adic we can also use the corresponding $\mathrm{WittVector}_\mathcal{O}(\square)$ to define the corresponding $z$-adic de Rham period ring which gives rise to $z$-adic \textit{prismatization} after \cite{BS2}, \cite{BLA}, \cite{DC}. However here we only use $p$-adic prismatization.
\end{remark}

\newpage

\section{Stackification}

\noindent In this section we prove the following results:

\begin{theorem}
$\mathrm{Stack}_{\mathrm{Bun},G,X,n}(\square)$ for $n\geq 2$ are actually equal to there stackification over the $v$-sites. And they satisfy the condition of being of smallness. Finally we can prove that they are satisfying the condition of being Artin as in \cite{FS}.
\end{theorem}

\begin{theorem}
\begin{align}
&\mathrm{FilFibCat}_{\mathrm{crys},n}(\square),\\
&\mathrm{FilFibCat}_{\mathrm{crys},n,\mathrm{Frob}}(\square),
\end{align} 
are stacks, with certain morphisms into:
\begin{align}
\mathrm{Stack}_{\mathrm{Bun},G,X,n}.
\end{align}
\end{theorem}

\begin{theorem}
Over the stack $\mathrm{Stack}_{\mathrm{Bun},G,X,n}$ we have the corresponding $v$-stacks of categories of prismatic quasi-coherent sheaves:
\begin{align}
\mathrm{QuasiCohSheaves}^\blacksquare_{\mathrm{prismatic},\mathrm{Stack}_{\mathrm{prismatic}}}(\mathrm{Stack}_{\mathrm{Bun},G,X,n}).
\end{align}
These quasi-coherent sheaves can be used to define the $p$-adic motives over category $\mathrm{Stack}_{\mathrm{Bun},G,X,n}$. Moreover we have the well-defined category of all the objects have the requirement $\mathrm{End}$ is a line:
\begin{align}
\mathrm{QuasiCohSheaves}^\blacksquare_{\mathrm{prismatic},\mathrm{Stack}_{\mathrm{prismatic}}}(\mathrm{Stack}_{\mathrm{Bun},G,X,n}).
\end{align}
They cover the objects related to lisse $p$-adic representations of the reductive groups.
\end{theorem}

\begin{theorem}
Over the stack $\mathrm{Stack}_{\mathrm{Bun},G,X,n}$ we have the corresponding $v$-stacks of categories of prismatic quasi-coherent sheaves:
\begin{align}
\mathrm{QuasiCohSheaves}^\blacksquare_{\mathrm{prismatic},\mathrm{Stack}_{\mathrm{prismatic},\mathrm{deRham}}}(\mathrm{Stack}_{\mathrm{Bun},G,X,n}).
\end{align}
These quasi-coherent sheaves can be used to define the $p$-adic motives over category $\mathrm{Stack}_{\mathrm{Bun},G,X,n}$. Moreover we have the well-defined category of all the objects have the requirement $\mathrm{End}$ is a line:
\begin{align}
\mathrm{QuasiCohSheaves}^\blacksquare_{\mathrm{prismatic},\mathrm{Stack}_{\mathrm{prismatic}}}(\mathrm{Stack}_{\mathrm{Bun},G,X,n}).
\end{align}
They cover the objects related to lisse $p$-adic representations of the reductive groups.
\end{theorem}

\begin{theorem}
There is a way to generalize \cite{FS} to our setting, namely for any $n$-fold covering of Galois group $\mathrm{Gal}_{X,n}$ by taking the roots of cyclotomic charaters up to order $n$, we have certain condensed parametrization:
\begin{align}
\mathrm{Gal}_{X,n} \rightarrow G^\mathrm{dual}(\overline{\mathbb{Q}}_\ell)\times \mathrm{Weil}_X
\end{align}
for generalized $\ell$-adic perverse sheaves attached to $G(X)$, which make $\mathrm{End}$ a line. One can also replace the full Galois group with the corresponding Weil group version. For any $n$-fold covering of Galois group $\mathrm{Gal}_{X,n}$ by taking the roots of cyclotomic charaters up to order $n$, we have certain condensed parametrization:
\begin{align}
\mathrm{Weil}_{X,n} \rightarrow G^\mathrm{dual}(\overline{\mathbb{Q}}_\ell)\times \mathrm{Weil}_X
\end{align}
for generalized $\ell$-adic perverse sheaves attached to $G(X)$, which make $\mathrm{End}$ a line. 
\end{theorem}

\subsection{Stackification for Generalized Hodge Modules}

\noindent Motivated by the conjectures in the section above, we consider the corresponding algebraic geometric approaches by using sophisticated stacks in the following.

\begin{definition}
$\mathrm{Stack}_{(\varphi,\Gamma_n), G^\mathrm{Lan}(\overline{T}),\Pi_X[\log(1+T)^{1/n}]}$ are stacks over rigid analytic spaces $\mathrm{RAS}_X$, parametrizing arithmetically all the $(\varphi,\Gamma_n)$-$G^\mathrm{Lan}(\overline{T})$-bundles over the generalized Robba ring $\Pi_X[\log(1+T)^{1/n}]$ in the imperfect setting.
\end{definition}

\indent Then one can follow \cite{FS}, \cite{DHKM}, \cite{EGH}, \cite{Z} to consider the corresponding coherent sheaves over these stacks by using the points coming from rigid analytic spaces, even in the derived sense by animating the corresponding rigid analytic affinoids from \cite{CSA}, \cite{CSB}, \cite{CSC}, where we use the same notation to denote the corresponding $\infty$-stacks.

\begin{definition}
$\mathrm{Stack}^\blacksquare_{(\varphi,\Gamma_n), G^\mathrm{Lan}(\overline{T}),\Pi_X[\log(1+T)^{1/n}]}$ are $\infty$-stacks over derived rigid analytic spaces $\mathrm{RAS}^\blacksquare_X$, parametrizing arithmetically all the $(\varphi,\Gamma_n)$-$G^\mathrm{Lan}(\overline{T})$-bundles over the generalized Robba ring
 \begin{align}
 \Pi_{X,\blacksquare}[\log(1+T)^{1/n}]
 \end{align}
 in the imperfect setting. We then have the corresponding coherent sheaves over these $\infty$-stacks, where we use the notation $\mathrm{Coh}_\blacksquare(\mathrm{Stack}^\blacksquare_{(\varphi,\Gamma_n), G^\mathrm{Lan}(\overline{T}),\Pi_X[\log(1+T)^{1/n}]})$ to denote the corresponding condesed coherent sheaves.
\end{definition}

There $\infty$-stacks are actually playing the roll of $\infty$-stacks in \cite{FS}, \cite{DHKM}, \cite{Z} in the $\ell$-setting, generalizing from \cite{EGH} to our mixed-parity situation. Namely they should be called  \textit{arithmetic stacks of $p$-adic Langlands Parametrization}.
 
\begin{definition}
$\mathrm{Stack}^\blacksquare_{(\varphi,\Gamma_n), G^\mathrm{Lan}(\overline{T}),\Pi_X[\log(1+T)^{1/n}]}$ are $\infty$-stacks over derived rigid analytic spaces $\mathrm{RAS}^\blacksquare_X$, parametrizing arithmetically all the $(\varphi,\Gamma_n)$-$G^\mathrm{Lan}(\overline{T})$-bundles over the generalized Robba ring 
\begin{align}
\Pi_{X,\blacksquare}[\log(1+T)^{1/n}]
\end{align}
in the imperfect setting. We then have the corresponding coherent sheaves over these $\infty$-stacks, where we use the notation $\mathrm{QuasiCoh}_\blacksquare(\mathrm{Stack}^\blacksquare_{(\varphi,\Gamma_n), G^\mathrm{Lan}(\overline{T}),\Pi_X[\log(1+T)^{1/n}]})$ to denote the corresponding condesed quasicoherent sheaves.
\end{definition}

\indent The program in \cite{FS} actually can be generalized to the mixed-parity situation in this paper in the following way, to use the corresponding stacks of \textit{geometric} family of relative Hodge structure in the $p$-adic and $z$-adic fashion.

\begin{definition}
We use the notation $\mathrm{TPERF}_{\mathrm{Spd}e}$ to denote all the Tate perfectoid spaces over $\mathrm{Spd}e$ as in \cite{FS}. $e/\mathbb{F}_p$ is assumed to be algebraically closure of $\mathbb{F}_p$. Then we consider the Fargues-Fontaine diamonds over this site carrying $v$-topology in the mixed-parity generalization situation. We then use the notation $\mathrm{FarFon}_{X,n}(\square)$ to denote the functors, where $\square$ is varying in the site $\mathrm{TPERF}_{\mathrm{Spd}e}$. To be more precise:
\begin{align}
\mathrm{FarFon}_{X,n}(\square) = \frac{\bigcup_{I} \mathrm{Spa}(\overline{\Pi}_{\square,X,I}[t^{1/n}]\otimes T, \overline{\Pi}_{+,\square,X,I}[t^{1/n}]\otimes T)}{\varphi}.
\end{align}
$t$ in the $p$-adic setting is defined to be \textit{functorially} $\log([1+\overline{z}])$ when we express $\square$ as Tate over $\mathbb{F}_p((\overline{z}))$. We require the Galois/Weil group acts on this chosen $t$ through $\chi_\mathrm{cyc}$.
\end{definition}

\begin{definition}
We use the notation $\mathrm{TPERF}_{\mathrm{Spd}e}$ to denote all the Tate perfectoid spaces over $\mathrm{Spd}e$ as in \cite{FS}. $e/\mathbb{F}_p$ is assumed to be algebraically closure of $\mathbb{F}_p$. Then we consider the Fargues-Fontaine diamonds over this site carrying $v$-topology in the mixed-parity generalization situation. We then use the notation $\mathrm{FarFon}_{X,n}(\square)$ to denote the functors, where $\square$ is varying in the site $\mathrm{TPERF}_{\mathrm{Spd}e}$. To be more precise:
\begin{align}
\mathrm{FarFon}_{X,n}(\square) = \frac{\bigcup_{I} \mathrm{Spa}(\overline{\Pi}_{\square,X,I}[t^{1/n}]\otimes T, \overline{\Pi}_{+,\square,X,I}[t^{1/n}]\otimes T)}{\varphi}.
\end{align}
$t$ in the $p$-adic setting is defined to be \textit{functorially} $\log([1+\overline{z}])$ when we express $\square$ as Tate over $\mathbb{F}_p((\overline{z}))$. We require the Galois/Weil group acts on this chosen $t$ through $\chi_\mathrm{cyc}$. Then over these diamonds we define the corresponding \textit{moduli analytic pre-$v$-stacks of $G$-bundles}. We use the corresponding notation $\mathrm{Stack}_{\mathrm{Bun},G,X,n}(\square)$ to denote the prestack evaluating over $\square$ certain groupoid of $G$-bundles over our stacks 
\begin{align}
\mathrm{FarFon}_{X,n}(\square) = \frac{\bigcup_{I} \mathrm{Spa}(\overline{\Pi}_{\square,X,I}[t^{1/n}]\otimes T, \overline{\Pi}_{+,\square,X,I}[t^{1/n}]\otimes T)}{\varphi}.
\end{align} 
\end{definition}

\begin{theorem}
$\mathrm{Stack}_{\mathrm{Bun},G,X,n}(\square)$ for $n\geq 2$ are actually equal to there stackification over the $v$-sites. And they satisfy the condition of being of smallness. Finally we can prove that they are satisfying the condition of being Artin as in \cite{FS}.
\end{theorem}

\begin{proof}
The corresponding proof on the $v$-stackification equal to the prestack itself is following on the corresponding fact that our Robba rings are actually finite over the corresponding usual Robba rings in \cite{FS}, \cite{KLA}, \cite{KLB}. Therefore the $v$-descent requirement for the $v$-stackification actually follows as in \cite{FS} on the corresponding $R\Gamma$ functors and the corresponding complexes. The smallness can be formally checked following \cite[See the corresponding 1.3 of Chapter III, the Proposition]{FS}. However the Artinness is not trivial. The idea is to use the corresponding Schburt varieties in \cite{FS} for $\mathrm{Stack}_{\mathrm{Bun},G,X,n}(\square)$ when $n=1$. In such a way we have the corresponding smooth presentation (actually being \textit{lisse cohomologically} from \cite{SchC}) from some diamond which needs to be locally required to be spatial. The foundation in \cite{SchC} realizes the stability of the corresponding pull-back mechanism for the corresponding \textit{lisse cohomologically} morphisms in the large category of $v$-stacks. Then we consider the corresponding smooth presentation for:
\begin{align}
\mathrm{Stack}_\mathrm{smooth} \longrightarrow \mathrm{Stack}_{\mathrm{Bun},G,X,1}(\square)
\end{align}
to achieve the corresponding smooth presentation for all the corresponding stacks:
\begin{align}
\mathrm{Stack}_{\mathrm{smooth}}\times \mathrm{Stack}_{\mathrm{Bun},G,X,n}(\square)  \longrightarrow \mathrm{Stack}_{\mathrm{Bun},G,X,n}(\square).
\end{align}
Finally the proof on the corresponding required diagonal morphisms can be proved in the same fashion by using the pull-backs along:
\begin{align}
\mathrm{Stack}_{\mathrm{Bun},G,X,n}\longrightarrow \mathrm{Stack}_{\mathrm{Bun},G,X,1},n\geq 2.
\end{align}
\end{proof}

\subsection{Prismatic Consideration}

\noindent We now construct some other moprhism to the stacks we established above following \cite{FS}. First for any ring $\square$ in $\mathrm{TPERF}_{\mathrm{Spd}e}$, one can consider the corresponding prismatic site:
\begin{align}
\square_\Delta,\mathcal{U}_{\Delta}
\end{align} 
after \cite{BS2}, where one can consider the corresponding crystalline crystal structure sheaf over this site as well. Moreover we consider the corresponding generalization in the following way. For distinguished ideal sheaf $I$, we consider the extension by using the corresponding square root of the distinguished element locally ${I}^{1/n}$ to form the corresponding generalized prismatic site:
\begin{align}
\square_{\Delta,n},\mathcal{U}_{\Delta,n}\otimes {O}_T
\end{align} 
Then we one can consider the corresponding crystalline crystal structure sheaf over this site as well. The structure sheaf will produce after taking the global section the corresponding crystalline ring $S_\mathrm{crystalline}$ in the style of $A$ and then $B$. This functionalization will produce the corresponding Frobenius endowed crystals and non-Frobenius endowed crystals, where we have correspnding filtration indexed by $\frac{1}{n}$ of the all the integers. We use the notation as in the below to denote the corresponding fiber categories:
\begin{align}
&\mathrm{FilFibCat}_{\mathrm{crys},n}(\square),\\
&\mathrm{FilFibCat}_{\mathrm{crys},n,\mathrm{Frob}}(\square).
\end{align}

After \cite{FS} we have the following theorem which provides certain \textit{stackification} for the prismatic crystalline crystals in the mixed-parity generalization:

\begin{theorem}
\begin{align}
&\mathrm{FilFibCat}_{\mathrm{crys},n}(\square),\\
&\mathrm{FilFibCat}_{\mathrm{crys},n,\mathrm{Frob}}(\square),
\end{align} 
are stacks, with certain morphisms into:
\begin{align}
\mathrm{Stack}_{\mathrm{Bun},G,X,n}.
\end{align}
\end{theorem}

Then we have the corresponding de Rham crystal consideration by looking at:
\begin{align}
\square_{\Delta,n},\mathcal{U}_{\Delta,n}[1/p]_{I^{1/n}}\otimes {O}_T, \mathcal{U}_{\Delta,n}[1/p]_{I^{1/n}}[1/I^{1/n}]\otimes {O}_T. 
\end{align} 
Here when we have $\square$ the corresponding prestacks of the corresponding categories of the corresponding de Rham crystals can be defined to be such a way to be stacks after \cite{CSA}, \cite{CSB}, \cite{CSC}, \cite{SchD}, \cite{SALBRC}, \cite{TA}, \cite{TB}, \cite{TC}, \cite{BBBK}. We use the notation 
\begin{align}
\mathrm{FibCat}_{\mathcal{U}_{\Delta,n}[1/p]_{I^{1/n}}[1/I^{1/n}]\otimes {O}_T}(\square) 
\end{align}
to denote the corresponding prestacks of the corresponding categories of the corresponding de Rham crystals.
\begin{theorem}
We have that the corresponding prestacks of de Rham crystals:
\begin{align}
\mathrm{FibCat}_{\mathcal{U}_{\Delta,n}[1/p]_{I^{1/n}}[1/I^{1/n}]\otimes {O}_T}(\square) 
\end{align}
are actually stacks in the $v$-topology, in our mixed-parity situation. \end{theorem}

\subsection{$p$-adic Motives}

Motives over $v$-stacks are challeging problems when it comes to study the corresponding Langlands conjectures in our current setting. Here after \cite{CSA}, \cite{CSB}, \cite{CSC}, \cite{SchD}, \cite{SALBRC}, \cite{TA}, \cite{TB}, \cite{TC}, \cite{BBBK}, \cite{BS2}, \cite{BLA}, \cite{DC} we use the analytic stackification approaches for the corresponding $p$-adic motives. The idea is as in the following. We start with our stack:
\begin{align}
\mathrm{Stack}_{\mathrm{Bun},G,X,n}.
\end{align}
The perfectoid charts for this $v$-stacks can be used to study the corresponding prismatic cohomology. For any $\square$ such perfectoid chart we have the corresponding prismatic stack:
\begin{align}
\mathrm{Stack}_{\mathrm{prismatic}}(\square)
\end{align}
with the corresponding quasicoherent sheaves over it:
\begin{align}
\mathrm{QuasiCohSheaves}_{\mathrm{prismatic},\mathrm{Stack}_{\mathrm{prismatic}}}(\square).
\end{align}
Varying the corresponding perfectoids we have the prestacks of categories:
\begin{align}
\mathrm{Stack}_{\mathrm{prismatic}}(\square)
\end{align}
with the corresponding quasicoherent sheaves over it:
\begin{align}
\mathrm{QuasiCohSheaves}_{\mathrm{prismatic},\mathrm{Stack}_{\mathrm{prismatic}}}(\square).
\end{align}
Then by taking the corresponding condensed analytification we have the corresponding analytic version of the stacks.
\begin{align}
\mathrm{Stack}^\blacksquare_{\mathrm{prismatic}}(\square)
\end{align}
with the corresponding \textit{solid} quasicoherent sheaves over it:
\begin{align}
\mathrm{QuasiCohSheaves}^\blacksquare_{\mathrm{prismatic},\mathrm{Stack}_{\mathrm{prismatic}}}(\square).
\end{align}

\begin{theorem}
Over the stack $\mathrm{Stack}_{\mathrm{Bun},G,X,n}$ we have the corresponding $v$-stacks of categories of prismatic quasi-coherent sheaves:
\begin{align}
\mathrm{QuasiCohSheaves}^\blacksquare_{\mathrm{prismatic},\mathrm{Stack}_{\mathrm{prismatic}}}(\mathrm{Stack}_{\mathrm{Bun},G,X,n}).
\end{align}
These quasi-coherent sheaves can be used to define the $p$-adic motives over category $\mathrm{Stack}_{\mathrm{Bun},G,X,n}$. Moreover we have the well-defined category of all the objects have the requirement $\mathrm{End}$ is a line:
\begin{align}
\mathrm{QuasiCohSheaves}^\blacksquare_{\mathrm{prismatic},\mathrm{Stack}_{\mathrm{prismatic}}}(\mathrm{Stack}_{\mathrm{Bun},G,X,n}).
\end{align}
They cover the objects related to lisse $p$-adic representations of the reductive groups.
\end{theorem}

\begin{proof}
Here we rely on the descent results from \cite{CSA}, \cite{CSB}, \cite{CSC} for general solid prismatic quasi-coherent sheaves.
\end{proof}

\indent For any $\square$ such perfectoid chart we have the corresponding prismatic stack:
\begin{align}
\mathrm{Stack}_{\mathrm{prismatic}, \mathrm{deRham}}(\square)
\end{align}
with the corresponding quasicoherent sheaves over it:
\begin{align}
\mathrm{QuasiCohSheaves}_{\mathrm{prismatic},\mathrm{Stack}_{\mathrm{prismatic}, \mathrm{deRham}}}(\square).
\end{align}
Varying the corresponding perfectoids we have the prestacks of categories:
\begin{align}
\mathrm{Stack}_{\mathrm{prismatic}, \mathrm{deRham}}(\square)
\end{align}
with the corresponding quasicoherent sheaves over it:
\begin{align}
\mathrm{QuasiCohSheaves}_{\mathrm{prismatic},\mathrm{Stack}_{\mathrm{prismatic},\mathrm{deRham}}}(\square).
\end{align}
Then by taking the corresponding condensed analytification we have the corresponding analytic version of the stacks.
\begin{align}
\mathrm{Stack}^\blacksquare_{\mathrm{prismatic},\mathrm{deRham}}(\square)
\end{align}
with the corresponding \textit{solid} quasicoherent sheaves over it:
\begin{align}
\mathrm{QuasiCohSheaves}^\blacksquare_{\mathrm{prismatic},\mathrm{Stack}_{\mathrm{prismatic},\mathrm{deRham}}}(\square).
\end{align}

\indent Here the corresponding de Rham stack over any perfectoid $\square$ is defined to be the formal spectrum of the corresponding de Rham period sheaf attached to an untilt $\sharp\square$ of $\square$, since we are working in an absolute situation where we do not fix some prism as the base ring. To be more precise we use the corresponding stack 
\begin{align}
\mathrm{Specformal}\Pi_\mathrm{deRham}(\sharp\square^\flat)=\mathrm{Specformal}\mathrm{WittVector}_\mathcal{O}(\sharp\square^\flat)[1/p]_{I_{\sharp\square}}. 
\end{align}
Then we take the corresponding \textit{condensed analytification}. Along this way we have many other stacks such as the corresponding crystalline and the corresponding semi-stable ones again after \cite{BS2}, \cite{BLA}, \cite{DC}:
\begin{align}
&\mathrm{Specformal}\Pi_\mathrm{cristalline}(\sharp\square^\flat),\\
&\mathrm{Specformal}\Pi_\mathrm{semistable}(\sharp\square^\flat).
\end{align}
If we have a fixed prism as the base ring then we can just use the Cartier divisor for this ring to take the completion with respect to this divisor after inverting $p$. In this manner if $X$ is $z$-adic we can also use the corresponding $\mathrm{WittVector}_\mathcal{O}(\square)$ to define the corresponding $z$-adic de Rham period ring which gives rise to $z$-adic \textit{prismatization} after \cite{BS2}, \cite{BLA}, \cite{DC}. However here we only use $p$-adic prismatization. If $X$ is $z$-adic we can also use the corresponding
\begin{align}
\mathrm{WittVector}_\mathcal{O}(\square)
\end{align} 
 to define the corresponding $z$-adic de Rham period ring which gives rise to $z$-adic \textit{prismatization} after \cite{BS2}, \cite{BLA}, \cite{DC}. Then we take the corresponding \textit{condensed analytification}. First we use the formal spectrum of $\mathrm{WittVector}_\mathcal{O}(\square)$ to define the $z$-adic prismatization. Then along this way we have many other stacks such as the corresponding crystalline and the corresponding semi-stable ones again after \cite{BS2}, \cite{BLA}, \cite{DC}:
\begin{align}
&\mathrm{Specformal}\Pi_\mathrm{cristalline}(\sharp\square^\flat),\\&\mathrm{Specformal}\Pi_\mathrm{semistable}(\sharp\square^\flat).
\end{align}
Then we will then have the corresponding $z$-adic \textit{prismatization} for any $z$-adic formal scheme. We then use the notation 
\begin{align}
&\mathrm{Stack}_{\mathrm{prismatic},\mathrm{deRham}},\\
&\mathrm{Stack}_{\mathrm{prismatic},\mathrm{cristalline}},\\
&\mathrm{Stack}_{\mathrm{prismatic},\mathrm{semistable}}
\end{align}
to denote the stacks with the condensed-prismatization in both $p$-adic and $z$-adic setting:
\begin{align}
&\mathrm{Stack}^\blacksquare_{\mathrm{prismatic},\mathrm{deRham}},\\&\mathrm{Stack}^\blacksquare_{\mathrm{prismatic},\mathrm{cristalline}},\\
&\mathrm{Stack}^\blacksquare_{\mathrm{prismatic},\mathrm{semistable}}.
\end{align}
Then we have the corresponding $\infty$-categories of solid quasicoherent sheaves over these stacks. All above in this definition are defined over $K$, or a formal scheme $S$ over $\mathcal{O}_X$, or certainly $z$-adic rigid analytic space $R$ for instance after \cite{SchD} and \cite{SALBRC}. Then we generalize these definitions to the mixed-parity setting after \cite{BS} for instance in de Rham, cristalline and semi-stable situations:

\begin{definition}
Now we generalize the definition here for mixed-parity modules after \cite{BS}. We only consider \textit{de Rham, cristalline and semi-stable situations} in this definition. Therefore we add the corresponding $n$-th root of $t$ into \textit{all} of the following prismatizations. We work over some $K$ which is stack in $v$-site and we assume $K$ is small. We assume $K$ is defined over $\mathrm{Spd}\mathcal{O}_X$, where we will differentiate the two different situations. We work over $K$ in $v$-topology as well. Here the corresponding de Rham stack over any perfectoid $\square$ is defined to be the formal spectrum of the corresponding de Rham period sheaf attached to an untilt $\sharp\square$ of $\square$, since we are working in an absolute situation where we do not fix some prism as the base ring. To be more precise we use the corresponding stack $\mathrm{Specformal}\Pi_{\mathrm{deRham},n}(\sharp\square^\flat)=\mathrm{Specformal}\mathrm{WittVector}_\mathcal{O}(\sharp\square^\flat)[1/p]_{I_{\sharp\square}}[t_{\sharp\square}^{1/n}]$. Then we take the corresponding \textit{condensed analytification}. Along this way we have the corresponding crystalline and the corresponding semi-stable ones again after \cite{BS2}, \cite{BLA}, \cite{DC}:
\begin{align}
&\mathrm{Specformal}\Pi_{\mathrm{cristalline},n}(\sharp\square^\flat),\\
&\mathrm{Specformal}\Pi_{\mathrm{semistable},n}(\sharp\square^\flat).
\end{align}
If we have a fixed prism as the base ring then we can just use the Cartier divisor for this ring to take the completion with respect to this divisor after inverting $p$. In this manner if $X$ is $z$-adic we can also use the corresponding $\mathrm{WittVector}_\mathcal{O}(\square)$ to define the corresponding $z$-adic de Rham period ring which gives rise to $z$-adic \textit{de Rham prismatization} after \cite{BS2}, \cite{BLA}, \cite{DC}. However here we only use $p$-adic prismatization. If $X$ is $z$-adic we can also use the corresponding $\mathrm{WittVector}_\mathcal{O}(\square)$ to define the corresponding $z$-adic de Rham period ring which gives rise to $z$-adic \textit{de Rham prismatization} after \cite{BS2}, \cite{BLA}, \cite{DC}. Then we take the corresponding \textit{condensed analytification}. Then along this way we have  the corresponding crystalline and the corresponding semi-stable ones again after \cite{BS2}, \cite{BLA}, \cite{DC} by adding $n$-th roots of the element $t_{\sharp\square}$:
\begin{align}
&\mathrm{Specformal}\Pi_{\mathrm{cristalline},n}(\sharp\square^\flat),\\
&\mathrm{Specformal}\Pi_{\mathrm{semistable},n}(\sharp\square^\flat).
\end{align}
We then use the notation 
\begin{align}
&\mathrm{Stack}_{\mathrm{prismatic},\mathrm{deRham},n},\\
&\mathrm{Stack}_{\mathrm{prismatic},\mathrm{cristalline},n},\\
&\mathrm{Stack}_{\mathrm{prismatic},\mathrm{semistable},n}
\end{align}
to denote the stacks with the condensed-prismatization in both $p$-adic and $z$-adic setting:
\begin{align}
&\mathrm{Stack}^\blacksquare_{\mathrm{prismatic},\mathrm{deRham},n},\\&\mathrm{Stack}^\blacksquare_{\mathrm{prismatic},\mathrm{cristalline},n},\\
&\mathrm{Stack}^\blacksquare_{\mathrm{prismatic},\mathrm{semistable},n}.
\end{align}
Then we have the corresponding $\infty$-categories of solid quasicoherent sheaves over these stacks. All above in this definition are defined over $K$. 
\end{definition}

\begin{theorem}
Over the stack $\mathrm{Stack}_{\mathrm{Bun},G,X,n}$ we have the corresponding $v$-stacks of categories of prismatic quasi-coherent sheaves:
\begin{align}
\mathrm{QuasiCohSheaves}^\blacksquare_{\mathrm{prismatic},\mathrm{Stack}_{\mathrm{prismatic},\mathrm{deRham}}}(\mathrm{Stack}_{\mathrm{Bun},G,X,n}).
\end{align}
These quasi-coherent sheaves can be used to define the $p$-adic motives over category $\mathrm{Stack}_{\mathrm{Bun},G,X,n}$. Moreover we have the well-defined category of all the objects have the requirement $\mathrm{End}$ is a line:
\begin{align}
\mathrm{QuasiCohSheaves}^\blacksquare_{\mathrm{prismatic},\mathrm{Stack}_{\mathrm{prismatic}}}(\mathrm{Stack}_{\mathrm{Bun},G,X,n}).
\end{align}
They cover the objects related to lisse $p$-adic representations of the reductive groups.
\end{theorem}

\begin{proof}
Here we rely on the descent results from \cite{CSA}, \cite{CSB}, \cite{CSC} for general solid prismatic quasi-coherent sheaves.
\end{proof}

\begin{theorem}
Now we assume that the corresponding $X$ is $z$-adic. And we assume the corresponding de Rham stack in this theorem is $z$-adic as well. Over the stack $\mathrm{Stack}_{\mathrm{Bun},G,X,n}$ we have the corresponding $v$-stacks of categories of prismatic quasi-coherent sheaves:
\begin{align}
\mathrm{QuasiCohSheaves}^\blacksquare_{\mathrm{prismatic},\mathrm{Stack}_{\mathrm{prismatic},\mathrm{deRham}}}(\mathrm{Stack}_{\mathrm{Bun},G,X,n}).
\end{align}
These quasi-coherent sheaves can be used to define the $p$-adic motives over category $\mathrm{Stack}_{\mathrm{Bun},G,X,n}$. Moreover we have the well-defined category of all the objects have the requirement $\mathrm{End}$ is a line:
\begin{align}
\mathrm{QuasiCohSheaves}^\blacksquare_{\mathrm{prismatic},\mathrm{Stack}_{\mathrm{prismatic}}}(\mathrm{Stack}_{\mathrm{Bun},G,X,n}).
\end{align}
They cover the objects related to lisse $z$-adic representations of the reductive groups\footnote{Again we don't know what is the $z$-adic local Langlands correspondence but we conjecture this is way to define that in a geometrized fashion.}.
\end{theorem}

\begin{proof}
Here we rely on the descent results from \cite{CSA}, \cite{CSB}, \cite{CSC} for general solid prismatic quasi-coherent sheaves.
\end{proof}

\subsection{Construction of Generalized Langlands Correspondence}

\noindent We now construct generalized Langlands parametrization for the following conjecture:

\begin{theorem}
There is a way to generalize \cite{FS} to our setting, namely for any $n$-fold covering of Galois group $\mathrm{Gal}_{X,n}$ by taking the roots of cyclotomic charaters up to order $n$, we have certain condensed parametrization:
\begin{align}
\mathrm{Gal}_{X,n} \rightarrow G^\mathrm{dual}(\overline{\mathbb{Q}}_\ell)\times \mathrm{Weil}_X
\end{align}
for generalized $\ell$-adic perverse sheaves attached to $G(X)$, which make $\mathrm{End}$ a line. One can also replace the full Galois group with the corresponding Weil group version. For any $n$-fold covering of Galois group $\mathrm{Gal}_{X,n}$ by taking the roots of cyclotomic charaters up to order $n$, we have certain condensed parametrization:
\begin{align}
\mathrm{Weil}_{X,n} \rightarrow G^\mathrm{dual}(\overline{\mathbb{Q}}_\ell)\times \mathrm{Weil}_X
\end{align}
for generalized $\ell$-adic perverse sheaves attached to $G(X)$, which make $\mathrm{End}$ a line. 
\end{theorem}

\begin{proof}
As in \cite[Chapter IX 4.1, Chapter VIII 4.1]{FS} and after \cite{VL}, we consider the corresponding category of all the condesed $\overline{\mathbb{Q}}_\ell$-perverse sheaves in the sense of \cite{FS}:
\begin{align}
\mathrm{Category}^\blacksquare_{\mathrm{Stack}_{\mathrm{Bun},G,X,n},\overline{\mathbb{Q}}_\ell}
\end{align}
which produces the corresponding $\mathbb{Z}_\ell$-linear categoricalization in the formalism in in \cite[Chapter IX 4.1, Chapter VIII 4.1]{FS} and \cite{VL}. To fit into the formalism we need the corresponding action of Weil group $\mathrm{Weil}_{X,n}$ which acts through the action on all the distinguished elements $t$ including our joined ones. Another thing is that we need the corresponding action of the Langlands dual group on this category which is through pull back of the corresponding Hecke operator directly to our stacks. Namely we have two key essential morphisms in our situation generalizing the situation of \cite{FS}. The first one is the following one from the Hecke stack in our setting to the corresponding stack of $G$-bundles:
\begin{align}
\mathrm{Stack}_{\mathrm{Hecke},G,\{1\}}\longrightarrow \mathrm{Stack}_{\mathrm{Bun},G,X,n}
\end{align}
while the second one is the following one:
\begin{align}
\mathrm{Stack}_{\mathrm{Hecke},G,\{1\}}\longrightarrow \mathrm{Stack}_{\mathrm{Bun},G,X,n}\times \mathrm{Stack}_{\mathrm{Cartier},G,n}
\end{align}
the latter is the mixed-parity Cartier divior stack which parametrizes the corresponding distinguished primitive ideal in the mixed-parity Fargues-Fontaine curve. In our setting this stack further maps (by sending $t^{1/n}$ to the corresponding equivalence class for the $n$-fold covering of the Weil group) to the classifying stack of the $n$-fold covering of the Weil group, this map immediately realizes through pull back the corresponding image living in the $\mathrm{Weil}_{X,n}$-equivariant sub categories of all the perverse sheaves in the $\ell$-adic condensed categoricalization from \cite{FS}. Putting all these together we are now exactly in the corresponding context and formalism from \cite{FS} and \cite{VL} as mentioned above. This finishes the corresponding proof. Another way to prove this can also be achieved such as in \cite[Chapter I]{FS} while we consider the stack of shtukas living over our extended covering stacks by taking the corresponding stack of shtukas having one leg from \cite{FS} along:
\begin{align}
\mathrm{Stack}_{\mathrm{Bun},G,X,n} \rightarrow \mathrm{Stack}_{\mathrm{Bun},G,X,1}
\end{align}
which produces the stack of generalized shtukas:
\begin{align}
\mathrm{Stack}_{\mathrm{shtukas},G,X,b,n,\{1\}}, n\geq 2.
\end{align}
Then the corresponding push forward along the structure morphism of this stack will generated the desired complex in our setting to realize the corresponding excursion operators as in \cite{FS}, which produces the corresponding map from the stack of Langlands paraters cocycle spaces to the Bernstein center.
\end{proof}

\newpage
\section{Generalized geometrized $p$-adic Langlands Program}

\subsection{Generalized geometrized Breuil-Schneider parametrization}

Motives over $v$-stacks are challeging problems when it comes to study the corresponding Langlands conjectures in our current setting. Here after \cite{CSA}, \cite{CSB}, \cite{CSC}, \cite{SchD}, \cite{SALBRC}, \cite{TA}, \cite{TB}, \cite{TC}, \cite{BBBK}, \cite{BS2}, \cite{BLA}, \cite{DC} we use the analytic stackification approaches for the corresponding $p$-adic motives. The idea is as in the following. We start with our stack:
\begin{align}
\mathrm{Stack}_{\mathrm{Bun},G,X,n}.
\end{align}
The perfectoid charts for this $v$-stacks can be used to study the corresponding prismatic cohomology. For any $\square$ such perfectoid chart we have the corresponding prismatic stack:
\begin{align}
\mathrm{Stack}_{\mathrm{prismatic}}(\square)
\end{align}
with the corresponding quasicoherent sheaves over it:
\begin{align}
\mathrm{QuasiCohSheaves}_{\mathrm{prismatic},\mathrm{Stack}_{\mathrm{prismatic}}}(\square).
\end{align}
Varying the corresponding perfectoids we have the prestacks of categories:
\begin{align}
\mathrm{Stack}_{\mathrm{prismatic}}(\square)
\end{align}
with the corresponding quasicoherent sheaves over it:
\begin{align}
\mathrm{QuasiCohSheaves}_{\mathrm{prismatic},\mathrm{Stack}_{\mathrm{prismatic}}}(\square).
\end{align}
Then by taking the corresponding condensed analytification we have the corresponding analytic version of the stacks.
\begin{align}
\mathrm{Stack}^\blacksquare_{\mathrm{prismatic}}(\square)
\end{align}
with the corresponding \textit{solid} quasicoherent sheaves over it:
\begin{align}
\mathrm{QuasiCohSheaves}^\blacksquare_{\mathrm{prismatic},\mathrm{Stack}_{\mathrm{prismatic}}}(\square).
\end{align}

\begin{theorem}
Over the stack $\mathrm{Stack}_{\mathrm{Bun},G,X,n}$ we have the corresponding $v$-stacks of categories of prismatic quasi-coherent sheaves:
\begin{align}
\mathrm{QuasiCohSheaves}^\blacksquare_{\mathrm{prismatic},\mathrm{Stack}_{\mathrm{prismatic}}}(\mathrm{Stack}_{\mathrm{Bun},G,X,n}).
\end{align}
These quasi-coherent sheaves can be used to define the $p$-adic motives over category $\mathrm{Stack}_{\mathrm{Bun},G,X,n}$. Moreover we have the well-defined category of all the objects have the requirement $\mathrm{End}$ is a line:
\begin{align}
\mathrm{QuasiCohSheaves}^\blacksquare_{\mathrm{prismatic},\mathrm{Stack}_{\mathrm{prismatic}}}(\mathrm{Stack}_{\mathrm{Bun},G,X,n}).
\end{align}
They cover the objects related to lisse $p$-adic representations of the reductive groups.
\end{theorem}

\begin{proof}
Here we rely on the descent results from \cite{CSA}, \cite{CSB}, \cite{CSC} for general solid prismatic quasi-coherent sheaves.
\end{proof}

\noindent We then consider the corresponding contruction which is sort of geometrization of Breuil-Schneider conjecture in \cite{BS}. In \cite{BS}, the smooth representation of $G(X)$ in p-adic setting tensored with finite dimensional weights of $G(X)$ are expected to produce the expected $p$-adic Langlands parametrization, with key parameter as a pair from the Banach data as above. However all of these can be regarded as the corresponding sheaves in our setting.
In this category:
\begin{align}
\mathrm{QuasiCohSheaves}^\blacksquare_{\mathrm{prismatic},\mathrm{Stack}_{\mathrm{prismatic}}}(\mathrm{Stack}_{\mathrm{Bun},G,X,n}).
\end{align}
we have the corresponding objects where $\mathrm{End}$ is a line:
\begin{align}
\mathrm{QuasiCohSheaves}^\blacksquare_{\mathrm{prismatic},\mathrm{Stack}_{\mathrm{prismatic}}}(\mathrm{Stack}_{\mathrm{Bun},G,X,n})_{\mathrm{End}}.
\end{align}
We also have the corresponding weights sheaves which are just local system:
\begin{align}
\mathrm{QuasiCohSheaves}^\blacksquare_{\mathrm{prismatic},\mathrm{Stack}_{\mathrm{prismatic}}}(\mathrm{Stack}_{\mathrm{Bun},G,X,n})_{\mathrm{Loc}}.
\end{align}
Taking the corresponding solid tensor product we should end up the following conjecture:
\begin{conjecture}
The Breuil-Schneider style functor 
\begin{align}
&\otimes^\blacksquare: \mathrm{QuasiCohSheaves}^\blacksquare_{\mathrm{prismatic},\mathrm{Stack}_{\mathrm{prismatic}}}(\mathrm{Stack}_{\mathrm{Bun},G,X,n})_{\mathrm{End}}\\
&\times \mathrm{QuasiCohSheaves}^\blacksquare_{\mathrm{prismatic},\mathrm{Stack}_{\mathrm{prismatic}}}(\mathrm{Stack}_{\mathrm{Bun},G,X,n})_{\mathrm{Loc}}\\
& \rightarrow \mathrm{QuasiCohSheaves}^\blacksquare_{\mathrm{prismatic},\mathrm{Stack}_{\mathrm{prismatic}}}(\mathrm{Stack}_{\mathrm{Bun},G,X,n})
\end{align}
can be used to construct geometrization of the Breuil-Schneider paramatrization in \cite{BS}.
\end{conjecture}

\subsection{Generalized Langlands Parameter Stackification}

\indent After \cite{EGH} we immediately have the following stacks:

\begin{definition}
$\mathrm{Stack}_{(\varphi,\Gamma_n), G^\mathrm{Lan}(\overline{T}),\Pi_X[\log(1+T)^{1/n}]}$ are stacks over rigid analytic spaces $\mathrm{RAS}_X$, parametrizing arithmetically all the $(\varphi,\Gamma_n)$-$G^\mathrm{Lan}(\overline{T})$-bundles over the generalized Robba ring $\Pi_X[\log(1+T)^{1/n}]$ in the imperfect setting.
\end{definition}

\indent Then one can follow \cite{FS}, \cite{DHKM}, \cite{EGH}, \cite{Z} to consider the corresponding coherent sheaves over these stacks by using the points coming from rigid analytic spaces, even in the derived sense by animating the corresponding rigid analytic affinoids from \cite{CSA}, \cite{CSB}, \cite{CSC}, where we use the same notation to denote the corresponding $\infty$-stacks.

\begin{definition}
$\mathrm{Stack}^\blacksquare_{(\varphi,\Gamma_n), G^\mathrm{Lan}(\overline{T}),\Pi_X[\log(1+T)^{1/n}]}$ are $\infty$-stacks over derived rigid analytic spaces $\mathrm{RAS}^\blacksquare_X$, parametrizing arithmetically all the $(\varphi,\Gamma_n)$-$G^\mathrm{Lan}(\overline{T})$-bundles over the generalized Robba ring 
\begin{align}
\Pi_{X,\blacksquare}[\log(1+T)^{1/n}]
\end{align}
in the imperfect setting. We then have the corresponding coherent sheaves over these $\infty$-stacks, where we use the notation $\mathrm{Coh}_\blacksquare(\mathrm{Stack}^\blacksquare_{(\varphi,\Gamma_n), G^\mathrm{Lan}(\overline{T}),\Pi_X[\log(1+T)^{1/n}]})$ to denote the corresponding condesed coherent sheaves.
\end{definition}

One can generalize many stacks of Hodge structures to this setting as well. For instance one can have the corresponding stacks of $(\nabla,\Gamma_n)$-modules, again after \cite{EGH}. 

\begin{definition}
$\mathrm{Stack}_{(\varphi,\Gamma_n,\nabla), G^\mathrm{Lan}(\overline{T}),\Pi_X[\log(1+T)^{1/n}]}$ are stacks over rigid analytic spaces $\mathrm{RAS}_X$, paramet
rizing arithmetically all the $(\varphi,\Gamma_n,\nabla)$-$G^\mathrm{Lan}(\overline{T})$-bundles over the generalized Robba ring $\Pi_X[\log(1+T)^{1/n}]$ in the imperfect setting.
\end{definition}

\begin{definition}
$\mathrm{Stack}^\blacksquare_{(\varphi,\Gamma_n,\nabla), G^\mathrm{Lan}(\overline{T}),\Pi_X[\log(1+T)^{1/n}]}$ are $\infty$-stacks over derived rigid analytic spaces $\mathrm{RAS}^\blacksquare_X$, parametrizing arithmetically all the $(\varphi,\Gamma_n,\nabla)$-$G^\mathrm{Lan}(\overline{T})$-bundles over the generalized Robba ring $\Pi_{X,\blacksquare}[\log(1+T)^{1/n}]$ in the imperfect setting. We then have the corresponding coherent sheaves over these $\infty$-stacks, where we use the notation $\mathrm{Coh}_\blacksquare(\mathrm{Stack}^\blacksquare_{(\varphi,\Gamma_n,\nabla), G^\mathrm{Lan}(\overline{T}),\Pi_X[\log(1+T)^{1/n}]})$ to denote the corresponding condesed coherent sheaves.
\end{definition}

\newpage
\subsection*{Acknowledgements}
Langlands conjecture and the related mathematics are quite deep. We acquired much suggestive motivation on the Langlands program and all the corresponding relevant knowledge we elaborated even indirectly in this paper, from Professor Sorensen and Professor Kedlaya. We are thankful to them to have gained these from them.


\newpage
\begin{thebibliography}{}

\bibitem[DHKM]{DHKM} Dat, Jean-Fran\c{c}ois, David Helm, Robert Kurinczuk, and Gilbert Moss. "Moduli of Langlands parameters." arXiv preprint arXiv:2009.06708 (2020).

\bibitem[TA]{TA} Tong, Xin. "Topologization and Functional Analytification I: Intrinsic Morphisms of Commutative Algebras." arXiv preprint arXiv:2102.10766 (2021).

\bibitem[La]{La} R. Langlands. 1967. Letter to Andr\'e Weil.

\bibitem[DA]{DA} Drinfel'd, Vladimir G. "Elliptic modules." Mathematics of the USSR-Sbornik 23, no. 4 (1974): 561.

\bibitem[DB]{DB} Drinfeld, Vladimir Gershonovich. "Langlands' conjecture for $GL(2)$ over functional fields." In Proceedings of the International Congress of Mathematicians (Helsinki, 1978), vol. 2, pp. 565-574. 1980.

\bibitem[A]{A} Abe, Tomoyuki. "Langlands correspondence for isocrystals and the existence of crystalline companions for curves." Journal of the American Mathematical Society 31, no. 4 (2018): 921-1057.

\bibitem[VL]{VL} Lafforgue, Vincent. "Chtoucas pour les groupes r\'eductifs et param\'etrisation de Langlands globale." Journal of the American Mathematical Society 31, no. 3 (2018): 719-891.

\bibitem[DK]{DK} Drinfeld, Vladimir, and Kiran S. Kedlaya. "Slopes of indecomposable $ F $-isocrystals." Pure and Applied Mathematics Quarterly 13, no. 1 (2017): 131-192.

\bibitem[XZ]{XZ} Xu, Daxin, and Xinwen Zhu. "Bessel F-isocrystals for reductive groups." Inventiones mathematicae 227, no. 3 (2022): 997-1092.

\bibitem[GL]{GL} Genestier, Alain, and Vincent Lafforgue. "Chtoucas restreints pour les groupes r\'eductifs et param\'etrisation de Langlands locale." arXiv preprint arXiv:1709.00978 (2017).

\bibitem[K1]{K1} Kedlaya, Kiran S. "Drinfeld's Lemma for F-isocrystals, I." International Mathematics Research Notices (2024): rnae039. 

\bibitem[KX]{KX} Kedlaya, Kiran S., and Daxin Xu. "Drinfeld's lemma for F-isocrystals, II: Tannakian approach." Compositio Mathematica 160, no. 1 (2024): 90-119.

\bibitem[BS]{BS} Breuil, Christophe and Schneider, Peter. "First steps towards p-adic Langlands functoriality" Journal f\"ur die reine und angewandte Mathematik, vol. 2007, no. 610, 2007, pp. 149-180. https://doi.org/10.1515/CRELLE.2007.070 
  
 
\bibitem[FS]{FS} Fargues, Laurent, and Peter Scholze. "Geometrization of the local Langlands correspondence." arXiv preprint arXiv:2102.13459 (2021).

\bibitem[KLA]{KLA} Kedlaya, Kiran S., and Ruochuan Liu. "Relative p-adic Hodge theory: foundations." arXiv preprint arXiv:1301.0792 (2013). Ast\'erisque 2015.

\bibitem[KLB]{KLB} Kedlaya, Kiran S., and Ruochuan Liu. "Relative p-adic Hodge theory, II: Imperfect period rings." arXiv preprint arXiv:1602.06899 (2016).

\bibitem[SchA]{SchA} Scholze, Peter. "Perfectoid spaces." Publications math\'ematiques de l'IH\'ES 116, no. 1 (2012): 245-313.

\bibitem[SchB]{SchB} Scholze, Peter. "P-adic Hodge theory for rigid-analytic varieties." In Forum of Mathematics, Pi, vol. 1, p. e1. Cambridge University Press, 2013.

\bibitem[SchC]{SchC} Scholze, Peter. "\'Etale cohomology of diamonds." arXiv preprint arXiv:1709.07343 (2017).

\bibitem[Ta]{Ta} Tate, John T. "p-Divisible groups." In Proceedings of a Conference on Local Fields: NUFFIC Summer School held at Driebergen (The Netherlands) in 1966, pp. 158-183. Berlin, Heidelberg: Springer Berlin Heidelberg, 1967.

\bibitem[FA]{FA} Fontaine, Jean-Marc. "Sur certains types de repr\'esentations p-adiques du groupe de Galois d'un corps local; construction d'un anneau de Barsotti-Tate." Annals of Mathematics 115, no. 3 (1982): 529-577.

\bibitem[CSA]{CSA} Dustin Clausen and Peter Scholze. "Lectures on Condensed Mathematics." https://www.math.uni-bonn.de/people/scholze/Condensed.pdf.

\bibitem[CSB]{CSB} Dustin Clausen and Peter Scholze. "Lectures on Analytic Geometry." https://www.math.uni-bonn.de/people/scholze/Analytic.pdf.

\bibitem[CSC]{CSC} Dustin Clausen and Peter Scholze. "Analytic Stacks." https://people.mpim-bonn.mpg.de/scholze/AnalyticStacks.html.

\bibitem[BBBK]{BBBK} Bambozzi, Federico, Oren Ben-Bassat, and Kobi Kremnizer. "Analytic geometry over F1 and the Fargues-Fontaine curve." Advances in Mathematics 356 (2019): 106815.

\bibitem[BS2]{BS2} Bhatt, Bhargav, and Peter Scholze. "Prisms and prismatic cohomology." Annals of Mathematics 196, no. 3 (2022): 1135-1275.

\bibitem[BLA]{BLA} Bhatt, Bhargav, and Jacob Lurie. "Absolute prismatic cohomology." arXiv preprint arXiv:2201.06120 (2022).

\bibitem[DC]{DC} Drinfeld, Vladimir. "Prismatization." arXiv preprint arXiv:2005.04746 (2020).

\bibitem[SchD]{SchD} Peter Scholze. "Some remarks on prismatic cohomology of rigid spaces." Talk notes.

\bibitem[SALBRC]{SALBRC} Peter Scholze, Johannes Ansch\"utz, Author-C\'esar Le Bras and Juan Estaban Rodriguez Camargo. "Analytic Prismatization." https://www.mpim-bonn.mpg.de/node/12943.

\bibitem[TB]{TB} Tong, Xin. "Topologization and Functional Analytification II: $\infty $-Categorical Motivic Constructions for Homotopical Contexts." arXiv preprint arXiv:2112.12679 (2021).

\bibitem[TC]{TC} Tong, Xin. "Topologization and Functional Analytification III." arXiv preprint arXiv:2405.14180 (2024).

\bibitem[Z]{Z} Zhu, Xinwen. "Coherent sheaves on the stack of Langlands parameters." arXiv preprint arXiv:2008.02998 (2020). 

\bibitem[EGH]{EGH} Emerton, Matthew, Toby Gee, and Eugen Hellmann. "An introduction to the categorical p-adic Langlands program." arXiv preprint arXiv:2210.01404 (2022).

\bibitem[KPX]{KPX} Kedlaya, Kiran, Jonathan Pottharst, and Liang Xiao. "Cohomology of arithmetic families of $(\varphi, \Gamma)$-modules." Journal of the American Mathematical Society 27, no. 4 (2014): 1043-1115.

\bibitem[CA]{CA} Colmez, Pierre. "Repr\'esentations de $GL_2(Qp)$ et $(\varphi, \Gamma)$-modules." Ast\'erisque 330, no. 281 (2010): 509.

\end{thebibliography}
\end{document}